\documentclass[11pt, noadjust]{amsart}
\usepackage[dvipsnames]{xcolor}
\usepackage[utf8]{inputenc}
\usepackage{amssymb,amsmath,amsthm,amsfonts,fullpage,float,latexsym,bbm,microtype,cite,fancyvrb}
\usepackage{tikz}
\usepackage{enumerate}
\usetikzlibrary{decorations.pathreplacing}
\usepackage{hyperref}
\hypersetup{colorlinks=true, linkcolor=blue, citecolor=magenta, filecolor=magenta, urlcolor=magenta}
\usepackage{bm}
\allowdisplaybreaks[1]

\makeatletter
\newtheorem*{rep@theorem}{\rep@title}\newcommand{\newreptheorem}[2]{%
\newenvironment{rep#1}[1]{%
\def\rep@title{\bf #2 \ref{##1}}%
\begin{rep@theorem}}%
{\end{rep@theorem}}}
\makeatother
\newreptheorem{theorem}{Theorem}

\newtheorem{theorem}{Theorem}
\newtheorem{proposition}[theorem]{Proposition}

\newtheorem{conjecture}[theorem]{Conjecture}
\newtheorem{lemma}[theorem]{Lemma}
\newtheorem{corollary}[theorem]{Corollary}
\theoremstyle{definition}
\newtheorem{remark}[theorem]{Remark}
\newtheorem{definition}[theorem]{Definition}

\newtheorem{example}[theorem]{Example}

\renewcommand{\sl}{\mathfrak{sl}}

\newcommand{\Q}{\mathbb{Q}}

\DeclareMathOperator{\Hilb}{Hilb}
\DeclareMathOperator{\Frob}{Frob}

\DeclareMathOperator{\Char}{Char}
\DeclareMathOperator{\GL}{GL}
\DeclareMathOperator{\SL}{SL}

\DeclareMathOperator{\Sym}{Sym}

\DeclareMathOperator{\sgn}{sgn}

\DeclareRobustCommand{\qbinom}{\genfrac[]{0pt}{}}

\begin{document}

\title{The sign character of the triagonal fermionic coinvariant ring}

\author[J. Lentfer]{John Lentfer}
\address{Department of Mathematics\\
         University of California, Berkeley, CA, USA}
\email{jlentfer@berkeley.edu}

\begin{abstract}
We determine the trigraded multiplicity of the sign character of the triagonal fermionic coinvariant ring $R_n^{(0,3)}$. 
As a corollary, this proves a conjecture of Bergeron (2020) that the multiplicity of the sign character of $R_n^{(0,3)}$ is $n^2-n+1$.
We also give an explicit formula for double hook characters in the diagonal fermionic coinvariant ring $R_n^{(0,2)}$, and discuss methods towards calculating the sign character of $R_n^{(0,4)}$.
Finally, we give a multigraded refinement of a conjecture of Bergeron (2020) that the multiplicity of the sign character of the $(1,3)$-bosonic-fermionic coinvariant ring $R_n^{(1,3)}$ is $\frac{1}{2}F_{3n}$, where $F_n$ is a Fibonacci number.
\end{abstract}

\maketitle

\section{Introduction}
The diagonal coinvariant ring 
\begin{equation} 
R_n^{(2,0)} := \Q[x_1,\ldots,x_n,y_1,\ldots,y_n]/\allowbreak \langle \Q[x_1,\ldots,x_n,y_1,\ldots,y_n]_+^{\mathfrak{S}_n} \rangle
\end{equation} 
was introduced by Haiman in 1994 \cite{Haiman1994}, and since then has been studied extensively. Its defining ideal is generated by all polynomials in $\Q[x_1,\ldots,x_n,y_1,\ldots,y_n]$, with no constant term, which are invariant under the diagonal action of $\mathfrak{S}_n$:
\begin{equation}
    \sigma \cdot p(x_1,\ldots,x_n,y_1,\ldots,y_n) = p(x_{\sigma(1)},\ldots,x_{\sigma(n)},y_{\sigma(1)},\ldots,y_{\sigma(n)}).
\end{equation}
Haiman found the dimension, bigraded Hilbert series, and bigraded Frobenius series of $R_n^{(2,0)}$ \cite{Haiman2002}.

There has been much recent interest (see \cite{BergeronOPAC, Bergeron2020, BHIR, Zabrocki2020}) in studying a more general class of coinvariant rings $R_n^{(k,j)}$ with $k$ sets of $n$ commuting (bosonic) variables $\bm{x}_n := \{x_1,\ldots, x_n\}$, $\bm{y}_n := \{y_1,\ldots, y_n\}$, $\bm{z}_n := \{z_1,\ldots,z_n\}$, etc., and $j$ sets of $n$ anticommuting (fermionic) variables $\bm{\theta}_n := \{\theta_1,\ldots, \theta_n\}$, $\bm{\xi}_n := \{\xi_1,\ldots, \xi_n\}$, $\bm{\rho}_n := \{\rho_1,\ldots,\rho_n\}$, etc.
We define the $(k,j)$-bosonic-fermionic coinvariant ring by \begin{equation} R_n^{(k,j)} := \Q[\underbrace{\bm{x}_n, \bm{y}_n, \bm{z}_n, \ldots}_k, \underbrace{\bm{\theta}_n, \bm{\xi}_n, \bm{\rho}_n, \ldots}_j]/\langle \Q[\underbrace{\bm{x}_n, \bm{y}_n, \bm{z}_n, \ldots}_k, \underbrace{\bm{\theta}_n, \bm{\xi}_n, \bm{\rho}_n, \ldots}_j]_+^{\mathfrak{S}_n} \rangle,\end{equation}
where its defining ideal is generated by all polynomials in $\Q[\underbrace{\bm{x}_n, \bm{y}_n, \bm{z}_n, \ldots}_k, \underbrace{\bm{\theta}_n, \bm{\xi}_n, \bm{\rho}_n, \ldots}_j]$, without constant term, which are invariant under the diagonal action of the symmetric group $\mathfrak{S}_n$, given by permuting the indices of the variables.
Commuting variables commute with all variables. Anticommuting variables anticommute with all anticommuting variables. That is, $\theta_i \theta_j = - \theta_j \theta_i$ for all $i,j$, and mixed products between different sets of fermionic variables likewise anticommute. Note that this implies that $\theta_i^2 = 0$. 

Note that $R_n^{(k,j)}$ also can be defined in terms of symmetric and exterior algebras:
\begin{equation}
    R_n^{(k,j)} = ((\Sym\Q^n)^{\otimes k} \otimes (\wedge\Q^n)^{\otimes j})/\langle((\Sym\Q^n)^{\otimes k} \otimes (\wedge\Q^n)^{\otimes j})^{\mathfrak{S}_n}_+ \rangle.
\end{equation}

We recall the definitions of the multigraded Hilbert and Frobenius series of $R_n^{(k,j)}$ (see for example \cite{Bergeron2020}). For fixed integers $k,j \geq 0$,
$R_n^{(k,j)}$ decomposes as a direct sum of multihomogeneous components, which are $\mathfrak{S}_n$-modules:
\begin{equation}R_n^{(k,j)} = \bigoplus_{r_1,\ldots,r_k, s_1, \ldots, s_j \geq 0} (R_n^{(k,j)})_{r_1,\ldots,r_k, s_1, \ldots, s_j}.\end{equation}
We denote the multigraded Hilbert series by 
\begin{equation}
\begin{aligned} 
\Hilb&(R_n^{(k,j)}; q_1, \ldots, q_k; u_1,\ldots, u_j)\\ &:= \sum_{r_1,\ldots,r_k, s_1, \ldots, s_j \geq 0}\dim \left((R_n^{(k,j)})_{r_1,\ldots,r_k, s_1, \ldots, s_j} \right)q_1^{r_1} \cdots q_k^{r_k}u_1^{s_1} \cdots u_j^{s_j},
\end{aligned}
\end{equation}
and the multigraded Frobenius series by 
\begin{equation}
\begin{aligned} 
\Frob&(R_n^{(k,j)}; q_1, \ldots, q_k; u_1,\ldots, u_j)\\ &:= \sum_{r_1,\ldots,r_k, s_1, \ldots, s_j \geq 0} F\Char\big((R_n^{(k,j)})_{r_1,\ldots,r_k, s_1, \ldots, s_j} \big)q_1^{r_1} \cdots q_k^{r_k}u_1^{s_1} \cdots u_j^{s_j},
\end{aligned}
\end{equation}
where $F$ denotes the Frobenius characteristic map and $\Char$ denotes the character.
For simplicity, if $k \leq 2$, we will use $q,t$ for $q_1,q_2$ and if $j \leq 4$, we will use $u,\allowbreak v,\allowbreak w,\allowbreak z$ for $u_1,\allowbreak u_2,\allowbreak u_3,\allowbreak u_4$. Recall that $\langle \Frob(R_n^{(k,j)}; q_1, \ldots, q_k; u_1,\ldots, u_j), h_1^n \rangle =\allowbreak \Hilb(R_n^{(k,j)}; q_1, \ldots, q_k; u_1,\ldots, u_j)$.
Furthermore, $R_n^{(k,j)}$ is a $\GL_k \times \GL_j \times \mathfrak{S}_n$-module (see \cite[Section 2]{Bergeron2020}), so its multigraded Frobenius character is a sum of products of three Schur functions, which are irreducible characters of polynomial representations of $\GL_k$ and $\GL_j$, along with a Frobenius character. The main focus of this paper is on the case $(k,j) = (0,3)$. 

Observe that the Schur function $s_{(\ell-2)}(u,v,w)$ is a $u,v,w$-analogue of the binomial coefficient $\binom{\ell}{2}$; denote it by
\begin{equation} \binom{\ell}{2}_{u,v,w} := s_{(\ell-2)}(u,v,w).\end{equation}
Recall the notation $[n]_{u,v} := u^{n-1} + u^{n-2}v + \cdots + uv^{n-2} + v^{n-1}$. It follows that as a polynomial in $w$ with coefficients in $u$ and $v$ we have
\begin{align}
    \binom{\ell}{2}_{u,v,w} & = [\ell-1]_{u,v} + [\ell-2]_{u,v}w + [\ell-3]_{u,v}w^2 + \cdots + [1]_{u,v}w^{\ell-2},
\end{align} and there are two similar expressions for $\binom{\ell}{2}_{u,v,w}$ as a polynomial in $u$ or in $v$.

Our main result is the following.

\begin{theorem}\label{thm:main-theorem}
\begin{equation}
    \begin{aligned}
        \langle \Frob(R_n^{(0,3)};u,v,w), s_{(1^n)}\rangle &= s_{(n-1)}(u,v,w)+s_{(n-2,1,1)}(u,v,w)\\
        &= \binom{n+1}{2}_{u,v,w} + uvw\binom{n-1}{2}_{u,v,w}.
    \end{aligned}
\end{equation}
\end{theorem}

The theorem immediately implies the following corollary, which was conjectured by Bergeron \cite[Table 3]{Bergeron2020}.

\begin{corollary}
    \begin{equation}\langle \Frob(R_n^{(0,3)};1,1,1), s_{(1^n)}\rangle = n^2-n+1.\end{equation}
\end{corollary}

\begin{proof}
    Recall that $\binom{\ell}{2}_{u,v,w}|_{u=v=w=1} = \binom{\ell}{2}$. By evaluating Theorem~\ref{thm:main-theorem} at $u=v=w=1$, the multiplicity is $\binom{n+1}{2} + \binom{n-1}{2} = n^2-n+1$.
\end{proof}

The organization of the paper is as follows. In Section~\ref{sec:coinvs-har}, we detail the setting of $R_n^{(0,3)}$ from the perspective of both coinvariants and harmonics. 
In Section~\ref{sec:background-preliminaries}, we recall some results of Haglund--Sergel \cite{HaglundSergel} and Kim--Rhoades \cite{KimRhoades2022}, along with proving some preliminary results. 
In Section~\ref{sec:upper-bound}, building on the work of Haglund--Sergel and Kim--Rhoades, we give an upper bound on the multiplicity of the sign character of $R_n^{(0,3)}$ (Corollary~\ref{cor:upper-bound}). 
In Section~\ref{sec:construction}, we construct two elements in the ring of triagonal fermionic harmonics $T_n$ (Proposition~\ref{prop:harmonic}), and study the $\GL_3$-representations that they generate (Proposition~\ref{prop:existence}), which constructs enough elements in $T_n$ to show that the upper bound on the multiplicity of the sign character is achieved with equality, proving the main theorem. 

In Section~\ref{sec:double_hook}, we derive a formula for double hook characters of $R_n^{(0,2)}$ (Theorem~\ref{thm:double-hook}).
In Section~\ref{sec:four_fermions}, we discuss methods to analyze the sign character of $R_n^{(0,4)}$.
In Section~\ref{sec:one-three}, we provide a $q,u,v,w$-refinement of a conjecture of Bergeron \cite{Bergeron2020} on the sign character of $R_n^{(1,3)}$ (Conjecture~\ref{conj:one-three}).

\section{Coinvariants and Harmonics}\label{sec:coinvs-har}

We now describe the setting in more detail, from the perspective of both coinvariants and, isomorphically, harmonics.

Specialize to $R_{n}^{(0,3)} := \Q[\bm{\theta}_n, \bm{\xi}_n, \bm{\rho}_n]/\langle \Q[\bm{\theta}_n, \bm{\xi}_n, \bm{\rho}_n]_+^{\mathfrak{S}_n}\rangle$, which we call the \textbf{triagonal fermionic coinvariant ring}. Its defining ideal $\langle \Q[\bm{\theta}_n, \bm{\xi}_n, \bm{\rho}_n]_+^{\mathfrak{S}_n}\rangle$ is the ideal generated by polynomials in $\Q[\bm{\theta}_n, \bm{\xi}_n, \bm{\rho}_n]$, without constant term, invariant under the diagonal action of $\mathfrak{S}_n$. A generating set for the ideal $\langle \Q[\bm{\theta}_n, \bm{\xi}_n, \bm{\rho}_n]_+^{\mathfrak{S}_n}\rangle$ is given by all the nonzero monomial symmetric functions in any of the three sets of variables:
\begin{equation}
\begin{aligned}
    \{ \theta_1+\cdots+\theta_n,\xi_1+\cdots+\xi_n,\rho_1+\cdots+\rho_n, \theta_1\xi_1+\cdots+\theta_n\xi_n,\\
    \theta_1\rho_1+\cdots+\theta_n\rho_n, \xi_1\rho_1+\cdots+\xi_n\rho_n, \theta_1\xi_1\rho_1+\cdots+\theta_n\xi_n\rho_n \}.
\end{aligned}
\end{equation}

Consider the ring $\Q[\bm{\theta}_n, \bm{\xi}_n, \bm{\rho}_n]$ in three sets of $n$ anticommuting variables, which we arrange into the $3 \times n$ matrix
\begin{equation}\label{eq:matrix} M :=\begin{bmatrix}
    \theta_1 & \theta_2 & \cdots & \theta_n\\
    \xi_1 & \xi_2 & \cdots & \xi_n\\
    \rho_1 & \rho_2 & \cdots & \rho_n
\end{bmatrix}.\end{equation}
For each $3 \times 3$ invertible matrix $A$ in $\GL_3$, multiply $M$ on the left by $A$ to obtain the product $A\cdot M$. 
Then, in any polynomial $f \in \Q[\bm{\theta}_n, \bm{\xi}_n, \bm{\rho}_n]$, replace each variable $\theta_i, \xi_i,\rho_i$ by the corresponding entry of the matrix $A\cdot M$. 
This defines a left $\GL_3$-action on $\Q[\bm{\theta}_n, \bm{\xi}_n, \bm{\rho}_n]$.

The ring $\Q[\bm{\theta}_n, \bm{\xi}_n, \bm{\rho}_n]$ is naturally trigraded.
Denote by $(\Q[\bm{\theta}_n, \bm{\xi}_n, \bm{\rho}_n])_{a,b,c}$ the homogeneous component spanned by all monomials that contain exactly $a$ variables in $\bm{\theta}_n$, $b$ variables in $\bm{\xi}_n$, and $c$ variables in $\bm{\rho}_n$.
Since $\GL_3$ acts by linear combinations of the three rows, the $\GL_3$-action preserves the total degree $d = a+b+c$. Hence each total degree $d$ component $\bigoplus_{a+b+c=d}(\Q[\bm{\theta}_n, \bm{\xi}_n, \bm{\rho}_n])_{a,b,c}$ is $\GL_3$-invariant.

Each permutation $\sigma \in \mathfrak{S}_n$ corresponds to a permutation matrix $P_\sigma$, defined by $(P_\sigma)_{i,j} = 1$ if $i = \sigma(j)$ and $0$ otherwise. Then the right multiplication $M \cdot P_\sigma$ has the effect on $M$ of letting $\sigma$ act on the indices of all variables: $\theta_i \mapsto\theta_{\sigma(i)}$, $\xi_i \mapsto\xi_{\sigma(i)}$, and $\rho_i \mapsto\rho_{\sigma(i)}$. Then, in any polynomial $f \in \Q[\bm{\theta}_n, \bm{\xi}_n, \bm{\rho}_n]$, replace each variable $\theta_i, \xi_i,\rho_i$ by the corresponding entry of the matrix $ M \cdot P_\sigma$. This defines a right $\mathfrak{S}_n$-action on $\Q[\bm{\theta}_n, \bm{\xi}_n, \bm{\rho}_n]$.

Because left multiplication by $A$ and right multiplication by $P_\sigma$ commute, these $\GL_3$ and $\mathfrak{S}_n$-actions commute, so $\Q[\bm{\theta}_n, \bm{\xi}_n, \bm{\rho}_n]$ is a $\GL_3 \times \mathfrak{S}_n$-module.

Recall the definition of derivative of anticommuting variables (where here each $\theta_i$ could be from any of the three sets of anticommuting variables) is (see for example \cite[Section 1.5]{SwansonWallach2}):
\begin{equation} \partial_{\theta_j} \theta_{i_1}\cdots\theta_{i_k} = \begin{cases}
    (-1)^{\ell - 1}\theta_{i_1}\cdots\hat{\theta}_{i_\ell}\cdots\theta_{i_k} &\text{ if } j = i_\ell,\\
    0 &\text{ otherwise.}
\end{cases}\end{equation}
Define the space of \textbf{triagonal fermionic harmonics} by
\begin{equation} 
T_n := \left\{ f \in \Q[\bm{\theta}_n, \bm{\xi}_n, \bm{\rho}_n]\, \Bigg| \, \sum_{i=1}^n \partial_{\theta_i}^h \partial_{\xi_i}^k \partial_{\rho_i}^\ell f = 0 \text{ for all } h+k+\ell > 0 \right\}.
\end{equation}
Since any $\theta_i^2 = 0$, we need not check any second derivatives, hence
\begin{equation}\label{eq:simplified-harmonics} T_n = \left\{ f \in \Q[\bm{\theta}_n, \bm{\xi}_n, \bm{\rho}_n]\, \Bigg| \, \sum_{i=1}^n \partial_{\theta_i}^h \partial_{\xi_i}^k \partial_{\rho_i}^\ell f = 0 \text{ for all } h+k+\ell > 0 \text{ and } h,k,\ell \in \{0,1\} \right\}.\end{equation}

Looking at the infinitesimal action of the Lie algebra $\sl_3$ on $T_n$ allows us to study $T_n$ as a $\GL_3$-module.\footnote{See for example \cite[Section 15.5]{FultonHarris} for more on the relationship between $\SL_3$ and $\GL_3$-representations.}
Similarly to \cite[Section 3.1]{Haiman1994}, we define the operators
\begin{equation}\label{eq:F_operators} F^{\theta \rightarrow \xi} := \sum_{i=1}^n \xi_i \partial_{\theta_i}, \text{ and } F^{\xi \rightarrow \rho} := \sum_{i=1}^n \rho_i \partial_{\xi_i},\end{equation}
and \begin{equation}\label{E_operators}E^{\theta \leftarrow \xi} := \sum_{i=1}^n \theta_i \partial_{\xi_i} \text{ and } E^{\xi \leftarrow \rho} := \sum_{i=1}^n \xi_i \partial_{\rho_i}.\end{equation}
Along with two $H$ operators given by taking the commutators of the $E$ and $F$ operators, these generate the Lie algebra $\sl_3$.

Furthermore $T_n$ is a $\GL_3\times\mathfrak{S}_n$-module. Write a trigraded component of $T_n$ as $(T_n)_{a,b,c}$ and $R_n^{(0,3)}$ as $(R_n^{(0,3)})_{a,b,c}$.
Then we have that $T_n \cong R_{n}^{(0,3)}$ as trigraded $\mathfrak{S}_n$-modules. Working with harmonics was advocated by Garsia: a benefit of working with harmonics over coinvariants is that one works with polynomials instead of equivalence classes.

A polynomial $p \in \Q[\bm{\theta}_n,\bm{\xi}_n,\bm{\rho}_n]$ is called antisymmetric if $\sigma(p) = \sgn(\sigma)p$ for all $\sigma \in \mathfrak{S}_n$, where $\sigma$ acts diagonally by permuting the variables.
Let $(T_n)^\epsilon$ denote the antisymmetric subspace\footnote{This is sometimes called the alternating subspace.} of $T_n$, which consists of all elements $p \in T_n$ which are antisymmetric.
This corresponds to the $u,v,w$-graded multiplicity of the sign character in $R_{n}^{(0,3)}$, that is, $\Hilb((T_n)^\epsilon; u,v,w) = \langle \Frob(R_n^{(0,3)};u,v,w), s_{(1^n)}\rangle$. 

\section{Preliminary results}\label{sec:background-preliminaries}

In this section, we recall or prove some preliminary results.
Haglund and Sergel gave a formula for the graded Frobenius series of the fermionic coinvariants $R_{n}^{(0,1)} := \Q[\bm{\theta}_n]/\langle \Q[\bm{\theta}_n]_+^{\mathfrak{S}_n}\rangle$.

\begin{lemma}[\!\!{{\cite[Lemma 4.10]{HaglundSergel}}}]\label{lem:HSlemma}
    For $n \geq 1$,
    \begin{equation} \Frob(R_n^{(0,1)};w) =  \sum_{k=0}^{n-1} w^k s_{(n-k,1^k)}.\end{equation}
\end{lemma}

Kim and Rhoades gave a formula for the bigraded Frobenius series of the diagonal fermionic coinvariants $R_{n}^{(0,2)} := \Q[\bm{\theta}_n, \bm{\xi}_n]/\langle \Q[\bm{\theta}_n, \bm{\xi}_n]_+^{\mathfrak{S}_n}\rangle$. 

\begin{theorem}[\!\!{\cite[Theorem 6.1]{KimRhoades2022}}]\label{thm:KR-frob} For $n \geq 1$,
    \begin{equation}\Frob(R_n^{(0,2)};u,v) = \sum_{0\leq i+j < n} u^iv^j\left(s_{(n-i,1^i)}*s_{(n-j,1^j)}-s_{(n-i+1,1^{i-1})}*s_{(n-j+1,1^{j-1})}\right),\end{equation}
where $*$ denotes the Kronecker product and $s_{(n+1,1^{-1})}$ is interpreted as $0$.
\end{theorem}

In particular, they found bigraded multiplicities for the trivial, sign, and hook characters.

\begin{proposition}[\!\!{\cite[Proposition 6.2]{KimRhoades2022}}]\label{prop:KR-sign}
In $R_{n}^{(0,2)}$, the multiplicity of the trivial character is
\begin{equation}\langle\Frob(R_{n}^{(0,2)}; u,v),s_{(n)}\rangle = 1,\end{equation}
the bigraded multiplicity of the sign character is
    \begin{equation}\langle\Frob(R_{n}^{(0,2)}; u,v), s_{(1^n)}\rangle = [n]_{u,v},\end{equation}
    and for $0 < k < n-1$, the bigraded multiplicity of a hook character is
        \begin{equation}\langle\Frob(R_{n}^{(0,2)}; u,v), s_{(n-k,1^k)}\rangle = [k+1]_{u,v} +uv[k]_{u,v}.\end{equation}
\end{proposition}

As a consequence of their result on the sign character, we conclude the following.

\begin{corollary}\label{cor:one-two}\
\begin{enumerate}
\item If a nonzero harmonic polynomial in $T_n \cap \Q[\bm{\theta}_n]$ is antisymmetric, then it must be of degree exactly $n-1$.
\item If a nonzero harmonic polynomial in $T_n \cap \Q[\bm{\theta}_n,\bm{\xi}_n]$ is antisymmetric, then it must be of degree exactly $n-1$.
\end{enumerate}
\end{corollary}

\begin{proof}
    The first claim is implied by the second. To show the second, notice that Proposition~\ref{prop:KR-sign} shows that the multiplicity of each component of the sign character is $n-1$. 
\end{proof}

There is a useful condition for when antisymmetrizing a polynomial results in zero.

\begin{lemma}\label{lem:new-lemma}
Let $p$ be a polynomial in $\Q[\bm{\theta}_n,\bm{\xi}_n,\bm{\rho}_n]$. If there exists a transposition $s \in \mathfrak{S}_n$ such that $s(p) = p$, then 
\begin{equation}
    \sum_{\sigma \in \mathfrak{S}_n} \sgn(\sigma) \sigma(p) = 0.
\end{equation}
\end{lemma}

\begin{proof}
Consider the symmetric group $\mathfrak{S}_n$ as a set, and partition it $\mathfrak{S}_n = \mathfrak{S}_n' \sqcup \mathfrak{S}_n''$ into two equal-sized subsets such that for each $\sigma' \in \mathfrak{S}_n'$, we have that $\sigma' \cdot s = \sigma''$, for some $\sigma'' \in \mathfrak{S}_n''$. Since $\sgn(s) =-1$, then 
\begin{align} \sum_{\sigma \in \mathfrak{S}_n} \sgn(\sigma) \sigma(p)
    &= \sum_{\sigma' \in \mathfrak{S}_n'} \sgn(\sigma') \sigma'(p) + \sum_{\sigma' \in \mathfrak{S}_n'} -\sgn(\sigma') \sigma'(p)=0.
    \end{align}
\end{proof}

As a consequence, we establish that an antisymmetric polynomial must use enough distinct indices of variables to be nonzero.

\begin{lemma}\label{lem:unique-indices}
If a polynomial in $\Q[\bm{\theta}_n,\bm{\xi}_n,\bm{\rho}_n]$ is antisymmetric, and for the variables that appear, at most $n-2$ unique indices are used, then the polynomial is identically 0.
\end{lemma}

\begin{proof}
Call the polynomial in question $r$.
Without loss of generality, say that the indices used in $r$ are $\{1,2,\ldots, n-2\}$.
Since $r$ is antisymmetric, it can be obtained by applying an antisymmetrizing operator $\sum_{\sigma \in \mathfrak{S}_n} \sgn(\sigma) \sigma$ to some polynomial $p \in \Q[\bm{\theta}_{n-2}, \bm{\xi}_{n-2}, \bm{\rho}_{n-2}]$. 
Then $p$ is invariant under the action of $s_{n-1}$. 
Apply Lemma~\ref{lem:new-lemma} to conclude that $r=0$.
\end{proof}

\section{An upper bound on degree}\label{sec:upper-bound}

In this section, we prove an upper bound on the trigraded degree of the sign character.
The following result is inspired by a similar result of Haglund--Sergel (on a different ring, $R_n^{(2,1)}$) \cite[Theorem 4.11]{HaglundSergel}. 

\begin{proposition}\label{prop:(0,2)(0,1)bound} For $n \geq 1$,
    \begin{equation} \langle \Frob(R_n^{(0,2)} \otimes R_n^{(0,1)}; u,v,w), s_{(1^n)}\rangle = \binom{n+1}{2}_{u,v,w} + uvw\binom{n-1}{2}_{u,v,w} + uv[n-1]_{u,v}.\end{equation}
\end{proposition}

\begin{proof}
    As noted in \cite[equation (4.9)]{HaglundSergel}, following from \cite{Bessenrodt}, for any $\mathfrak{S}_n$-modules $A$ and $B$ tensored under the diagonal action of $\mathfrak{S}_n$, we have that
    \begin{equation} \Frob(A \otimes B) = \sum_{\nu \vdash n} s_\nu \sum_{\lambda, \mu \vdash n} \langle \Frob(A), s_\lambda \rangle \langle \Frob(B), s_\mu \rangle \langle s_\lambda * s_\mu, s_\nu \rangle,\end{equation}
    where $\langle s_\lambda * s_\mu, s_\nu \rangle = g(\lambda, \mu, \nu)$ are the Kronecker coefficients. In our case, we are interested in 
    \begin{equation}
    \begin{aligned} 
    \langle \Frob&(R_n^{(0,2)} \otimes R_n^{(0,1)}; u,v,w), s_{(1^n)}\rangle\\ 
    &=  \sum_{\lambda, \mu \vdash n} \langle \Frob(R_n^{(0,2)};u,v), s_\lambda \rangle \langle \Frob(R_n^{(0,1)};w), s_\mu \rangle g(\lambda, \mu, 1^n)  \\
    &=  \sum_{\lambda \vdash n} \langle \Frob(R_n^{(0,2)};u,v), s_{\lambda'} \rangle \langle \Frob(R_n^{(0,1)};w), s_{\lambda} \rangle,
    \end{aligned}
    \end{equation}
    since $g(\lambda, \mu, 1^n) = \delta_{\mu, \lambda'}$. By Lemma~\ref{lem:HSlemma}, which states that only hook Schur functions appear in $\Frob(R_n^{(0,1)};w)$, we reduce to
    \begin{equation}
    \begin{aligned} \langle \Frob(R_n^{(0,2)} \otimes R_n^{(0,1)}; u,v,w), s_{(1^n)}\rangle 
    &=  \sum_{k=0}^{n-1} w^k \langle \Frob(R_n^{(0,2)};u,v), s_{(k+1,1^{n-k-1})} \rangle  \\
    &= \sum_{k=0}^{n-1} w^k ( [n-k]_{u,v} + uv[n-k-1]_{u,v}),
    \end{aligned}
    \end{equation}
    where we applied Proposition~\ref{prop:KR-sign}. After a bit of algebra, we obtain the claimed formula.
\end{proof}

Now the following result gives us an upper bound on the occurrences of the sign character in $R_n^{(0,3)}$. For multivariate polynomials $A$ and $B$, the notation $A \leq B$ means that $B-A$ is a sum of monomials with only nonnegative coefficients.

\begin{corollary}\label{cor:upper-bound} For $n \geq 1$,
    \begin{equation} \langle \Frob(R_n^{(0,3)}; u,v,w), s_{(1^n)}\rangle \leq \binom{n+1}{2}_{u,v,w} + uvw\binom{n-1}{2}_{u,v,w}.\end{equation} 
\end{corollary}

\begin{proof}
    First note that any occurrence of the sign character in $R_n^{(0,3)}$ is predicated on it occurring in $R_n^{(0,2)} \otimes R_n^{(0,1)}$, since $R_n^{(0,3)}$ is a quotient of $R_n^{(0,2)} \otimes R_n^{(0,1)}$ under the diagonal action of $\mathfrak{S}_n$ (see \cite{HaglundSergel}). 

    By permuting which sets of variables out of $\bm{\theta}_n,\bm{\xi}_n,\bm{\rho}_n$ are assigned to $R_n^{(0,2)}$ and to $R_n^{(0,1)}$ in Proposition~\ref{prop:(0,2)(0,1)bound}, we conclude that a sign character in $R_n^{(0,3)}$ must have trigraded multiplicity bounded above by 
    \begin{equation}\label{eq:first-triple}
    \binom{n+1}{2}_{u,v,w} + uvw\binom{n-1}{2}_{u,v,w} + uv[n-1]_{u,v},
    \end{equation}
    \begin{equation}\label{eq:second-triple}
    \binom{n+1}{2}_{u,v,w} + uvw\binom{n-1}{2}_{u,v,w} + uw[n-1]_{u,w},
    \end{equation}
    and 
    \begin{equation}\label{eq:third-triple}
    \binom{n+1}{2}_{u,v,w} + uvw\binom{n-1}{2}_{u,v,w} + vw[n-1]_{v,w}.
    \end{equation}
    Note that if a polynomial $A \geq 0$ satisfies $ A \leq uv[n-1]_{u,v}$, $A \leq uw[n-1]_{u,w}$, and $A \leq vw[n-1]_{v,w}$, then $A =0$. Hence by taking the bounds given by equations~(\ref{eq:first-triple}-\ref{eq:third-triple}) together, we conclude that the sign character in $R_n^{(0,3)}$ must have trigraded multiplicity bounded above by 
    \begin{equation}\binom{n+1}{2}_{u,v,w} + uvw\binom{n-1}{2}_{u,v,w}.\end{equation}
\end{proof}

\section{Construction of basis elements}\label{sec:construction}

Now we will work towards the proof of the main theorem, by constructing two explicit elements in $T_n$, which are highest weight vectors for certain $\GL_3$-representations. We ultimately show that the upper bound given in Corollary~\ref{cor:upper-bound} is obtained with equality. 

\begin{definition}
    For $n \geq 1$, define the \textbf{primary theta-seed} by
\begin{equation} \Delta_1(\bm{\theta}_n) := \sum_{\sigma \in \mathfrak{S}_n} \sgn(\sigma)\sigma(\theta_1\theta_2\cdots\theta_{n-1}).\end{equation}
\end{definition}

\begin{definition}
    For $n \geq 3$, define the \textbf{secondary theta-seed} by
\begin{equation} \Delta_2(\bm{\theta}_n) := \sum_{\sigma \in \mathfrak{S}_n} \sgn(\sigma)\sigma((\theta_1\xi_2\rho_2 + \theta_2\xi_1\rho_2 + \theta_2\xi_2\rho_1)\theta_3 \theta_4 \cdots \theta_{n-1}).\end{equation}
\end{definition}
When $n=1$ or $3$ in the definitions of the primary and secondary theta-seed, respectively, the empty product of $\theta_i$'s is interpreted as $1$.

We now prove a technical lemma.

\begin{lemma}\label{lem:commuting-opertors}\
\begin{enumerate}
    \item The antisymmetrization operator $\sum_{\sigma \in \mathfrak{S}_n}\sgn(\sigma)\sigma$ commutes with the differential operator $\sum_{i=1}^n\partial_{\theta_i}^h \partial_{\xi_i}^k \partial_{\rho_i}^\ell$.
    \item The antisymmetrization operator $\sum_{\sigma \in \mathfrak{S}_n}\sgn(\sigma)\sigma$ commutes with the operators $F^{\theta \rightarrow \xi}$, $F^{\xi \rightarrow \rho}$, $E^{\theta \leftarrow \xi}$, and $E^{\xi \leftarrow \rho}$.
\end{enumerate} 
\end{lemma}

\begin{proof}
First we show (1). Let $f \in \Q[\bm{\theta}_n,\bm{\xi}_n,\bm{\rho}_n]$. Observe that for any $\sigma \in \mathfrak{S}_n$, 
\begin{equation} \sum_{i=1}^n \partial_{\theta_i}^h \partial_{\xi_i}^k \partial_{\rho_i}^\ell = \sum_{i=1}^n \partial_{\sigma(\theta_i)}^h \partial_{\sigma(\xi_i)}^k \partial_{\sigma(\rho_i)}^\ell,\end{equation} since addition is commutative. Thus,
\begin{equation}
\begin{aligned}
    \sum_{i=1}^n \partial_{\theta_i}^h \partial_{\xi_i}^k \partial_{\rho_i}^\ell \left( \sum_{\sigma \in \mathfrak{S}_n} \sgn(\sigma) \sigma(f)\right) &= \sum_{\sigma \in \mathfrak{S}_n} \sgn(\sigma) \sum_{i=1}^n  \partial_{\theta_i}^h \partial_{\xi_i}^k \partial_{\rho_i}^\ell (\sigma(f))\\
    &= \sum_{\sigma \in \mathfrak{S}_n} \sgn(\sigma) \sum_{i=1}^n  \partial_{\sigma(\theta_i)}^h \partial_{\sigma(\xi_i)}^k \partial_{\sigma(\rho_i)}^\ell (\sigma(f))\\
    &= \sum_{\sigma \in \mathfrak{S}_n}\sgn(\sigma)\sigma\left( \sum_{i=1}^n \partial_{\theta_i}^h \partial_{\xi_i}^k \partial_{\rho_i}^\ell f\right),
\end{aligned}
\end{equation}
so the operators commute as claimed.

Next we show (2). For any $\sigma \in \mathfrak{S}_n$, 
\begin{equation} \sum_{i=1}^n \xi_i \partial_{\theta_i} = \sum_{i=1}^n \sigma(\xi_i) \partial_{\sigma(\theta_i)},\end{equation} since addition is commutative. Thus for $F^{\theta \rightarrow \xi} = \sum_{i=1}^n \xi_i \partial_{\theta_i}$,
\begin{equation}
\begin{aligned}
    \sum_{i=1}^n \xi_i \partial_{\theta_i}  \left( \sum_{\sigma \in \mathfrak{S}_n} \sgn(\sigma) \sigma(f)\right) &= \sum_{\sigma \in \mathfrak{S}_n} \sgn(\sigma) \sum_{i=1}^n  \xi_i \partial_{\theta_i} \sigma(f)\\
    &= \sum_{\sigma \in \mathfrak{S}_n} \sgn(\sigma)  \sum_{i=1}^n \sigma(\xi_i) \partial_{\sigma(\theta_i)} \sigma(f)\\
    &= \sum_{\sigma \in \mathfrak{S}_n}\sgn(\sigma)\sigma\left(\sum_{i=1}^n \xi_i \partial_{\theta_i} f\right),
\end{aligned}
\end{equation}
and similarly for the other operators.
\end{proof}

Now we show that $\Delta_1(\bm{\theta}_n)$ and $\Delta_2(\bm{\theta}_n)$ are harmonic.

\begin{proposition}\label{prop:harmonic}\
    \begin{enumerate}
        \item The primary theta-seed $\Delta_1(\bm{\theta}_n)$ is in $(T_n)^\epsilon$.
        \item The secondary theta-seed $\Delta_2(\bm{\theta}_n)$ is in $(T_n)^\epsilon$.
    \end{enumerate}
\end{proposition}

\begin{proof}
We show (1). Consider the primary theta-seed
\begin{equation} \Delta_1(\bm{\theta}_n) = \sum_{\sigma \in \mathfrak{S}_n} \sgn(\sigma)\sigma(\theta_1 \theta_2 \cdots \theta_{n-1}).\end{equation} 
By equation~(\ref{eq:simplified-harmonics}), to show it is in $T_n$, we only need to show that $\sum_{i=1}^n\partial_{\theta_i}^h \partial_{\xi_i}^k \partial_{\rho_i}^\ell \Delta_1(\bm{\theta}_n) = 0$ for $(h,k,\ell) \in \{(1,0,0),\allowbreak (0,1,0),\allowbreak (0,0,1),\allowbreak (1,1,0),\allowbreak (1,0,1),\allowbreak (0,1,1),\allowbreak (1,1,1)\}.$ Since only $\theta_i$ variables appear in $\Delta_1(\bm{\theta}_n)$, it only remains to check $(1,0,0)$. Consider $\theta_1\cdots \theta_{n-1}$, which is what we will antisymmetrize to get the primary theta-seed. By Lemma~\ref{lem:commuting-opertors}, the differential operator commutes with the antisymmetrization operator. We write
\begin{equation} \sum_{i=1}^n \partial_{\theta_i} \theta_1\cdots \theta_{n-1} = \sum_{i=1}^n (-1)^{i-1} \theta_1\cdots\hat{\theta}_i\cdots \theta_{n-1},\end{equation}
where each monomial on the right hand side is in only $n-2$ variables. 
By Lemma~\ref{lem:unique-indices}, this becomes $0$ upon antisymmetrization. This shows $\Delta_1(\bm{\theta}_n)$ is in $T_n$, and since it is constructed via an antisymmetrization operator, it is in $(T_n)^\epsilon$.

We show (2). For $n \geq 3$, consider the secondary theta-seed
\begin{equation} \Delta_2(\bm{\theta}_n) = \sum_{\sigma \in \mathfrak{S}_n} \sgn(\sigma)\sigma((\theta_1\xi_2\rho_2 + \theta_2\xi_1\rho_2 + \theta_2\xi_2\rho_1)\theta_3 \theta_4 \cdots \theta_{n-1}).\end{equation}
We must check that $\Delta_2(\bm{\theta}_n)$ is in $T_n$. We only need to show that $\sum_{i=1}^n\partial_{\theta_i}^h \partial_{\xi_i}^k \partial_{\rho_i}^\ell \Delta_2(\bm{\theta}_n) = 0$ for the following seven choices of $(h,k,\ell)$: \begin{equation}\{(1,0,0), (0,1,0), (0,0,1), (1,1,0), (1,0,1), (0,1,1), (1,1,1)\}.\end{equation}

Case 1: $(h,k,\ell) = (1,0,0)$. Let $f = (\theta_1\xi_2\rho_2 + \theta_2\xi_1\rho_2 + \theta_2\xi_2\rho_1)\theta_3 \theta_4 \cdots \theta_{n-1}$, which is what we will antisymmetrize to get the secondary theta-seed. For concision, write $\prod_{i=3}^{n-1} \theta_i$ for the ordered product $\theta_3\theta_4\cdots\theta_{n-1}$. We write that
\begin{equation}
\begin{aligned}
    \sum_{i=1}^n \partial_{\theta_i} f &= \partial_{\theta_1} f + \partial_{\theta_2} f + \sum_{i=3}^{n-1} \partial_{\theta_i} f\\
    &= \xi_2\rho_2 \prod_{i=3}^{n-1} \theta_i + (\xi_1\rho_2 + \xi_2\rho_1)\prod_{i=3}^{n-1} \theta_i\\
    &\quad+ \sum_{i=3}^{n-1} (-1)^{i} (\theta_1\xi_2\rho_2 + \theta_2\xi_1\rho_2 + \theta_2\xi_2\rho_1)\theta_3 \cdots \hat{\theta}_i \cdots \theta_{n-1}.
\end{aligned}
\end{equation}
Then upon antisymmetrization, both of the terms $\xi_2\rho_2 \prod_{i=3}^{n-1} \theta_i$ and $\sum_{i=3}^{n-1} (-1)^{i} (\theta_1\xi_2\rho_2 + \theta_2\xi_1\rho_2 + \theta_2\xi_2\rho_1)\theta_3 \cdots \hat{\theta}_i \cdots \theta_{n-1}$ will be $0$ due to Lemma~\ref{lem:unique-indices}, since they only contain $n-2$ indices for variables.
Since the simple transposition $s_1$ leaves the term $(\xi_1\rho_2 + \xi_2\rho_1)\prod_{i=3}^{n-1} \theta_i$ invariant, by Lemma~\ref{lem:new-lemma}, it becomes $0$ upon antisymmetrization.

Case 2: $(h,k,\ell) = (0,1,0)$. Before antisymmetrization, we write that
\begin{equation}
\begin{aligned}
    \sum_{i=1}^n \partial_{\xi_i} f &= \partial_{\xi_1} f + \partial_{\xi_2} f \\
    &= -\theta_2 \rho_2 \prod_{i=3}^{n-1} \theta_i - (\theta_1\rho_2 + \theta_2\rho_1)\prod_{i=3}^{n-1} \theta_i.
\end{aligned}
\end{equation}
By Lemma~\ref{lem:unique-indices}, $\theta_2 \rho_2 \prod_{i=3}^{n-1} \theta_i$ becomes 0 upon antisymmetrization. By Lemma~\ref{lem:new-lemma}, $(\theta_1\rho_2 + \theta_2\rho_1)\prod_{i=3}^{n-1} \theta_i$ becomes 0 upon antisymmetrization, since it is invariant under $s_1$.

Case 3: $(h,k,\ell) = (0,0,1)$. The same argument as in Case 2 applies.

Case 4: $(h,k,\ell) = (1,1,0)$. Before antisymmetrization, we write that
\begin{equation}
\begin{aligned}
    \sum_{i=1}^n \partial_{\theta_i}\partial_{\xi_i} f &= \partial_{\theta_1}\partial_{\xi_1} f + \partial_{\theta_2}\partial_{\xi_2} f  + \sum_{i=3}^{n-1} \partial_{\theta_i}\partial_{\xi_i} f \\
    &= 0 - \rho_1\prod_{i=3}^{n-1} \theta_i + 0,
\end{aligned}
\end{equation}
which becomes 0 upon antisymmetrization by Lemma~\ref{lem:unique-indices}.

Case 5: $(h,k,\ell) = (1,0,1)$. The same argument as in Case 4 applies.

Case 6: $(h,k,\ell) = (0,1,1)$. The same argument as in Case 4 applies.

Case 7: $(h,k,\ell) = (1,1,1)$. Before antisymmetrization, we write that
\begin{equation}
\begin{aligned}
    \sum_{i=1}^n \partial_{\theta_i}\partial_{\xi_i}\partial_{\rho_i} f &= \partial_{\theta_1}\partial_{\xi_1}\partial_{\rho_1} f + \partial_{\theta_2}\partial_{\xi_2}\partial_{\rho_2} f  + \sum_{i=3}^{n-1} \partial_{\theta_i}\partial_{\xi_i}\partial_{\rho_i} f \\
    &= 0,
\end{aligned}
\end{equation}
which remains 0 upon antisymmetrization.

This completes the proof that $\Delta_2(\bm{\theta}_n)$ is in $T_n$, and since it is constructed via an antisymmetrization operator, it is in $(T_n)^\epsilon$. 
\end{proof}

A priori, a harmonic polynomial in $T_n$ could equal zero if certain term cancellations occur. However, the following result demonstrates that this does not happen for $\Delta_1(\bm{\theta}_n)$ and $\Delta_2(\bm{\theta}_n)$.

\begin{proposition}\label{prop:nonzero}\
    \begin{enumerate}
        \item The primary theta-seed $\Delta_1(\bm{\theta}_n)$ is nonzero in $T_n$.
        \item The secondary theta-seed $\Delta_2(\bm{\theta}_n)$ is nonzero in $T_n$.
    \end{enumerate}
\end{proposition}

\begin{proof}
    We show (1). Consider in $T_n$, \begin{equation} \Delta_1(\bm{\theta}_n) = \sum_{\sigma \in \mathfrak{S}_n} \sgn(\sigma)\sigma(\theta_1 \theta_2 \cdots \theta_{n-1}).\end{equation}
    When the antisymmetrization operator $\sum_{\sigma \in \mathfrak{S}_n} \sgn(\sigma)\sigma$ is applied to $\theta_1 \theta_2 \cdots \theta_{n-1}$, consider which $\sigma$ will output the monomial $\theta_1 \theta_2 \cdots \theta_{n-1}$.
    Such a $\sigma$ must fix $n$, but $1,\ldots,n-1$ can be permuted. For any $\sigma$ in the symmetric group on letters $1,\ldots,n-1$, we have that $\sgn(\sigma)\sigma(\theta_1\theta_2 \cdots \theta_{n-1}) = \theta_1\theta_2 \cdots \theta_{n-1}$ (see for example \cite[Chapter III, Section 7.3, Proposition 5]{BourbakiAlgebraI}). Hence antisymmetrization produces $(n-1)!$ copies of $\theta_1 \theta_2 \cdots \theta_{n-1}$, so $\Delta_1(\bm{\theta}_n)$ is nonzero in $T_n$, as desired.    

    Next we show (2). Consider in $T_n$, \begin{equation} \Delta_2(\bm{\theta}_n) = \sum_{\sigma \in \mathfrak{S}_n} \sgn(\sigma)\sigma((\theta_1\xi_2\rho_2 + \theta_2\xi_1\rho_2 + \theta_2\xi_2\rho_1)\theta_3 \theta_4 \cdots \theta_{n-1}).\end{equation}
    When the operator $\sum_{\sigma \in \mathfrak{S}_n} \sgn(\sigma)\sigma$ is applied to $(\theta_1\xi_2\rho_2 + \theta_2\xi_1\rho_2 + \theta_2\xi_2\rho_1)\theta_3 \theta_4 \cdots \theta_{n-1}$, consider which $\sigma$ will include the monomial $\theta_1\xi_2\rho_2\theta_3 \theta_4 \cdots \theta_{n-1}$ in its output.
    Since the $\xi_i$ and $\rho_i$ have the same index, it must come from an antisymmetrization of $\theta_1\xi_2\rho_2\theta_3 \theta_4 \cdots \theta_{n-1}$ (and not of $\theta_2\xi_1\rho_2\theta_3 \theta_4 \cdots \theta_{n-1}$ nor $\theta_2\xi_2\rho_1\theta_3 \theta_4 \cdots \theta_{n-1}$). Thus such a $\sigma$ must fix $2$ and $n$, but $1,3,4,\ldots,n-1$ can be permuted. For any $\sigma$ in the symmetric group on letters $1,3,4,\ldots,n-1$, it follows that $\sgn(\sigma)\sigma(\theta_1\theta_3 \theta_4 \cdots \theta_{n-1}) = \theta_1\theta_3 \theta_4 \cdots \theta_{n-1}$. This implies that
    $\sgn(\sigma)\sigma(\theta_1\xi_2\rho_2\theta_3 \theta_4 \cdots \theta_{n-1}) = \theta_1\xi_2\rho_2\theta_3 \theta_4 \cdots \theta_{n-1}$. Hence antisymmetrization produces $(n-2)!$ copies of $\theta_1\xi_2\rho_2\theta_3 \theta_4 \cdots \theta_{n-1}$, so $\Delta_2(\bm{\theta}_n)$ is nonzero in $T_n$, as desired.
\end{proof}

Starting with $\Delta_1(\bm{\theta}_n)$ and $\Delta_2(\bm{\theta}_n)$, the operators defined in equation~(\ref{eq:F_operators}) $F^{\theta \rightarrow \xi}$ and $F^{\xi \rightarrow \rho}$ create representations of $\GL_3$.

\begin{proposition}\label{prop:existence}\
    \begin{enumerate}
        \item The primary theta-seed $\Delta_1(\bm{\theta}_n)$ is the highest weight vector for the representation of $\GL_3$ with character $s_{(n-1)}(u,v,w)$.
        \item The secondary theta-seed $\Delta_2(\bm{\theta}_n)$ is the highest weight vector for the representation of $\GL_3$ with character $s_{(n-2,1,1)}(u,v,w)$.
    \end{enumerate}
\end{proposition}

\begin{proof}
We show (1). Start with the primary theta-seed
\begin{equation} \Delta_1(\bm{\theta}_n) = \sum_{\sigma \in \mathfrak{S}_n} \sgn(\sigma)\sigma(\theta_1 \theta_2 \cdots \theta_{n-1}).\end{equation}
By Proposition~\ref{prop:harmonic}, it is in $(T_n)^\epsilon$, and by Proposition~\ref{prop:nonzero}, it is nonzero. 
Observe that $E^{\theta \leftarrow \xi}$ and $E^{\xi \leftarrow \rho}$ both kill $\Delta_1(\bm{\theta}_n)$. 
Thus $\Delta_1(\bm{\theta}_n)$ is a highest weight vector for the irreducible $\GL_3$-representation with highest weight $(n-1,0,0)$, where the weight of any $p \in T_n$ is given by $(\deg_\theta p, \deg_{\xi} p, \deg_\rho p)$ (see for example \cite[Section 14-15]{FultonHarris} for a reference on the Lie theory used).
This $\GL_3$-representation is in $(T_n)^\epsilon$ since Lemma~\ref{lem:commuting-opertors} implies that $(T_n)^\epsilon$ is closed under $F^{\theta \to \xi}$ and $F^{\xi \to \rho}$.
The $\GL_3$-character of this representation is $s_{(n-1)}(u,v,w)$.

We show (2). Consider the secondary theta-seed
\begin{equation} \Delta_2(\bm{\theta}_n) = \sum_{\sigma \in \mathfrak{S}_n} \sgn(\sigma)\sigma((\theta_1\xi_2\rho_2 + \theta_2\xi_1\rho_2 + \theta_2\xi_2\rho_1)\theta_3 \theta_4 \cdots \theta_{n-1}).\end{equation}
This is antisymmetric, as it is constructed using an antisymmetrization operator. 
By Corollary~\ref{cor:one-two}, there are no antisymmetric elements of degree $n$ solely in any one or two sets of variables. This implies that $E^{\theta \leftarrow \xi}$ and $E^{\xi \leftarrow \rho}$ both kill $\Delta_2(\bm{\theta}_n)$.
By Proposition~\ref{prop:harmonic}, it is in $(T_n)^\epsilon$, and by Proposition~\ref{prop:nonzero}, it is nonzero. 
Thus $\Delta_2(\bm{\theta}_n)$ is a highest weight vector for the irreducible $\GL_3$-representation with highest weight $(n-2,1,1)$.
Again, this $\GL_3$-representation is in $(T_n)^\epsilon$ by Lemma~\ref{lem:commuting-opertors}.
The $\GL_3$-character of this representation is $s_{(n-2,1,1)}(u,v,w)$.
\end{proof}

Now we can prove the main theorem.

\begin{reptheorem}{thm:main-theorem}
        \begin{equation}
            \begin{aligned}\langle \Frob(R_n^{(0,3)}; u,v,w), s_{(1^n)}\rangle &= s_{(n-1)}(u,v,w) + s_{(n-2,1,1)}(u,v,w)\\ 
        &= \binom{n+1}{2}_{u,v,w} + uvw\binom{n-1}{2}_{u,v,w}.\end{aligned}
        \end{equation}
\end{reptheorem}

\begin{proof}
    The representations constructed in Proposition~\ref{prop:existence} are sufficient to force the bound in Corollary~\ref{cor:upper-bound} to be achieved with equality. A bit of algebra verifies that the formulation in terms of Schur functions is equivalent to the formulation in terms of $u,v,w$-binomial coefficients.
\end{proof}

\begin{remark}
    While it is not necessarily true in general that the Lie algebra operators $F$ and $E$ will correspond to crystal operators, in the present case, it is true since every weight space for the representations studied in Proposition~\ref{prop:existence} is one-dimensional. The representation with highest weight $(n-1,0,0)$ has crystal structure isomorphic to a crystal of tableaux $\mathcal{B}_{(n-1)}$ and the representation with highest weight $(n-2,1,1)$ has crystal structure isomorphic to a crystal of tableaux $\mathcal{B}_{(n-2,1,1)}$ (see for example \cite[Chapter 3]{BumpSchilling}).
\end{remark}

\section{Double hook characters of the diagonal fermionic coinvariant ring}\label{sec:double_hook}

With Proposition~\ref{prop:KR-sign}, Kim and Rhoades gave explicit formulas for the trivial, sign, and hook characters of $R_n^{(0,2)}$. In this section, we extend the analysis to give an explicit formula for double hook characters of $R_n^{(0,2)}$, where a \textbf{double hook} is a partition $\lambda = (\lambda_1 \geq \lambda_2 \geq \lambda_3 \geq \cdots)$ such that $\lambda_3 \leq 2$ and $\lambda_2 \geq 2$. All characters of $R_n^{(0,2)}$ indexed by shapes not contained in a double hook have multiplicity $0$.

Rosas gave a combinatorial formula for the Kronecker coefficients $g(\lambda, (n-e,1^e), (n-f,1^f))$, for any shape $\lambda$. These Kronecker coefficients are only nonzero if $\lambda$ is contained in a double hook shape. Here, we recall her formulas for double hook shapes, rows, and hook shapes.

\begin{theorem}[\!\!{\cite[Theorem 3]{Rosas}}]\label{thm:rosas} Let $(n-e,1^e)$ and $(n-f,1^f)$ be hook shapes.
\begin{enumerate}
    \item If $\lambda$ is not contained in a double hook, then $g(\lambda, (n-e,1^e), (n-f,1^f)) = 0$.
    \item Let $\lambda = (\lambda_1,\lambda_2,2^\ell,1^{n-2\ell-\lambda_1-\lambda_2})$ where $\lambda_1 \geq \lambda_2 \geq 2$ be a double hook. Then
    \begin{equation}
    \begin{aligned}
        g(&\lambda, (n-e,1^e), (n-f,1^f))\\
        &= \chi(\lambda_2-1 \leq \tfrac{e+f-n+\lambda_1+\lambda_2}{2} \leq \lambda_1)\cdot \chi(|f-e| \leq n-2\ell-\lambda_1-\lambda_2)\\
        &+\chi(\lambda_2 \leq \tfrac{e+f-n+\lambda_1+\lambda_2+1}{2} \leq \lambda_1)\cdot \chi(|f-e| \leq n-2\ell-\lambda_1-\lambda_2+1),
    \end{aligned}
    \end{equation}
    where $\chi(P)$ is $1$ if the proposition $P$ is true and $0$ if it is false.
    \item If $\lambda = (n)$, then $g(\lambda, (n-e,1^e), (n-f,1^f)) = \chi(e=f)$.
    \item If $\lambda$ is a hook shape $(n-d,1^d)$, then 
    \begin{equation}g(\lambda, (n-e,1^e), (n-f,1^f)) = \chi(|e-f| \leq d)\cdot\chi(d \leq e+f \leq 2n-d-2).\end{equation}
\end{enumerate}
\end{theorem}

This implies that these coefficients are in $\{0,1,2\}$ \cite[Corollary 4]{Rosas}. We are ready to prove the following result, establishing a formula for double hook characters in $R_n^{(0,2)}$.

\begin{theorem}\label{thm:double-hook}
    Let $\lambda \vdash n$ be a double hook, that is, $\lambda = (\lambda_1,\lambda_2, 2^\ell,1^{n-2\ell-\lambda_1-\lambda_2})$ where $\lambda_1 \geq \lambda_2 \geq 2$. Let $m := n-2\ell-\lambda_1-\lambda_2 +2$.
    Then if $\lambda_1 = \lambda_2$, we have that
    \begin{equation}\label{eq:d=0}
        \langle \Frob(R_n^{(0,2)}; u,v), s_{\lambda}\rangle = (uv)^{\ell+\lambda_2 - 1}(uv[m-2]_{u,v}+ [m]_{u,v} + [m-1]_{u,v}).
    \end{equation}
    If $\lambda_1 > \lambda_2$, we have that
    \begin{equation}\label{eq:d_positive}
        \langle \Frob(R_n^{(0,2)}; u,v), s_{\lambda}\rangle = (uv)^{\ell+\lambda_2 - 1}(uv[m-1]_{u,v} + uv[m-2]_{u,v} + [m]_{u,v} + [m-1]_{u,v} ).
    \end{equation}
\end{theorem}

\begin{proof}
    Recall from Theorem~\ref{thm:KR-frob} that 
    \begin{equation}\Frob(R_n^{(0,2)};u,v) = \sum_{0\leq i+j < n} u^iv^j\left(s_{(n-i,1^i)}*s_{(n-j,1^j)}-s_{(n-i+1,1^{i-1})}*s_{(n-j+1,1^{j-1})}\right).\end{equation}
    Thus
    \begin{equation}
    \begin{aligned}
        \langle &\Frob(R_n^{(0,2)};u,v), s_\lambda \rangle\\ 
        &= \sum_{0\leq i+j < n} u^iv^j\left(g(\lambda,(n-i,1^i),(n-j,1^j))-g(\lambda,(n-i+1,1^{i-1}),(n-j+1,1^{j-1}))\right).
    \end{aligned}
    \end{equation}
    Since $\lambda = (\lambda_1,\lambda_2, 2^\ell,1^{n-2\ell-\lambda_1-\lambda_2})$ is a double hook, by Theorem~\ref{thm:rosas},
    \begin{equation}\label{eq:large-equations}
    \begin{aligned}
    g(&\lambda,(n-i,1^i),(n-j,1^j)) - g(\lambda,(n-i+1,1^{i-1}),(n-j+1,1^{j-1}))\\
        &= \left(\chi(\lambda_2-1 \leq \tfrac{i+j-n+\lambda_1+\lambda_2}{2} \leq \lambda_1) -\chi(\lambda_2-1 \leq \tfrac{i+j-n+\lambda_1+\lambda_2-2}{2} \leq \lambda_1) \right)\\
        &\qquad \cdot \chi(|j-i| \leq n-2\ell-\lambda_1-\lambda_2)\\
        &\quad+\left(\chi(\lambda_2 \leq \tfrac{i+j-n+\lambda_1+\lambda_2+1}{2} \leq \lambda_1) - \chi(\lambda_2 \leq \tfrac{i+j-n+\lambda_1+\lambda_2-1}{2} \leq \lambda_1) \right)\\
        &\quad\qquad \cdot \chi(|j-i| \leq n-2\ell-\lambda_1-\lambda_2+1).
    \end{aligned}
    \end{equation}
    Define $d:= \lambda_1 - \lambda_2$. We analyze a component of equation~(\ref{eq:large-equations}):
    \begin{equation}\chi(\lambda_2-1 \leq \tfrac{i+j-n+\lambda_1+\lambda_2}{2} \leq \lambda_1) -\chi(\lambda_2-1 \leq \tfrac{i+j-n+\lambda_1+\lambda_2-2}{2} \leq \lambda_1),\end{equation}
    which is equivalent to 
    \begin{equation}\label{eq:first_chi_seq}
        \chi(|i+j - (n-1)| \leq d+1) - \chi(|i+j - (n+1)| \leq d+1).
    \end{equation}
    Since the bigraded component $(R_n^{(0,2)})_{i,j} = 0$ whenever $i+j \geq n$, we exclude these cases from our analysis. Thus $|i+j - (n-1)| \leq d+1$ is satisfied if $i+j = n-1$ or $n-2$ (if $d \geq 0$), if $i+j = n-3$ (if $d \geq 1$), if $i+j = n-4$ (if $d \geq 2$), etc. On the other hand, $|i+j - (n+1)| \leq d+1$ is satisfied if $i+j = n-1$ (if $d \geq 1$), if $i+j = n-2$ (if $d \geq 2$), etc. Putting these together, we get that equation~(\ref{eq:first_chi_seq}) is 1 when $i+j \in \{n-1-d,n-2-d\}$ and 0 otherwise.

    We analyze when condition $\chi(|j-i| \leq n-2\ell-\lambda_1-\lambda_2)$ is satisfied. At $i+j = n-1-d$, the condition is true exactly for integers $j$ which satisfy
    \begin{equation}
        \ell + \lambda_2 \leq j \leq n-d-\ell-\lambda_2-1.
    \end{equation}
    Summing up $u^{n-1-d-j}v^j$ over such $j$ gives 
    \begin{equation}\label{eq:first}
        (uv)^{\ell+\lambda_2}[n-d-2\ell-2\lambda_2]_{u,v}.
    \end{equation} 
    At $i+j = n-2-d$, the condition is true exactly for integers $j$ which satisfy
    \begin{equation}
        \ell + \lambda_2 -1 \leq j \leq n-d-\ell-\lambda_2-1.
    \end{equation}
    Summing up $u^{n-2-d-j}v^j$ over such $j$ gives 
    \begin{equation}\label{eq:second}
        (uv)^{\ell+\lambda_2-1}[n-d-2\ell-2\lambda_2+1]_{u,v}.
    \end{equation}

    Next we analyze another component of equation~(\ref{eq:large-equations}):
    \begin{equation}
        \chi(\lambda_2 \leq \tfrac{i+j-n+\lambda_1+\lambda_2+1}{2} \leq \lambda_1) - \chi(\lambda_2 \leq \tfrac{i+j-n+\lambda_1+\lambda_2-1}{2} \leq \lambda_1),
    \end{equation}
    which is equivalent to 
    \begin{equation}\label{eq:second_chi_seq}
        \chi(|i+j - (n-1)| \leq d) - \chi(|i+j - (n+1)| \leq d).
    \end{equation}
    We have that $|i+j - (n-1)| \leq d$ is satisfied if $i+j = n-1$ (if $d \geq 0$), if $i+j = n-2$ (if $d \geq 1$), if $i+j = n-3$ (if $d \geq 2$), etc. On the other hand, $|i+j - (n+1)| \leq d$ is satisfied if $i+j = n-1$ (if $d \geq 2$), if $i+j = n-2$ (if $d \geq 3$), etc. Putting these together, when $d=0$, we get that equation~(\ref{eq:second_chi_seq}) is 1 when $i+j = n-1$ and 0 otherwise. When $d \geq 1$, we get that equation~(\ref{eq:second_chi_seq}) is 1 when $i+j \in \{n-d,n-1-d\}$ and 0 otherwise.

    We analyze when condition $\chi(|j-i| \leq n-2\ell-\lambda_1-\lambda_2+1)$ is satisfied. At $i+j = n-1-d$, the condition is true exactly for integers $j$ which satisfy
    \begin{equation}
        \ell + \lambda_2 -1 \leq j \leq n-d-\ell-\lambda_2.
    \end{equation}
    Summing up $u^{n-1-d-j}v^j$ over such $j$ gives
    \begin{equation}\label{eq:third}
    (uv)^{\ell+\lambda_2-1}[n-d-2\ell-2\lambda_2+2]_{u,v}.
    \end{equation}
    When $d \geq 1$, at $i+j = n-d$, the condition is true exactly for integers $j$ which satisfy
    \begin{equation}
        \ell + \lambda_2 \leq j \leq n-d-\ell-\lambda_2.
    \end{equation}
    Summing up $u^{n-d-j}v^j$ over such $j$ gives 
    \begin{equation}\label{eq:fourth}
        (uv)^{\ell+\lambda_2}[n-d-2\ell-2\lambda_2+1]_{u,v}.
    \end{equation}

    For concision, we use $m := n-2\ell-\lambda_1-\lambda_2 +2 = n-d - 2\ell - 2\lambda_2+2$. When $d = 0$, i.e., $\lambda_1 = \lambda_2$, summing equations~(\ref{eq:first}), (\ref{eq:second}), and (\ref{eq:third}) proves equation~(\ref{eq:d=0}). When $d \geq 1$, i.e., $\lambda_1 > \lambda_2$, summing equations~(\ref{eq:first}), (\ref{eq:second}), (\ref{eq:third}), and (\ref{eq:fourth}) proves equation~(\ref{eq:d_positive}).
    
\end{proof}

\section{Four sets of fermions}\label{sec:four_fermions}

A natural further question is to determine $\langle \Frob(R_n^{(0,j)};u_1,\ldots,u_j), s_{(1^n)}\rangle$ for $j \geq 4$, of which the simplest next case is $\langle \Frob(R_n^{(0,4)};u,v,w,z), s_{(1^n)}\rangle$. 
If the hook characters in $R_n^{(0,3)}$ are all known, then a similar process as in Proposition~\ref{prop:(0,2)(0,1)bound} using $\Frob(R_n^{(0,3)} \otimes R_n^{(0,1)}; u,v,w,z)$ can be attempted. Using Theorem~\ref{thm:rosas}, we can first bound $\langle \Frob(R_n^{(0,3)}; u,v,w), s_{(n-d,1^d)}\rangle$.

\begin{example}
Let $n=3$. Suppose we want to calculate $\langle \Frob(R_3^{(0,3)}; u,v,w), s_{(2,1)}\rangle$. Consider that any $s_{(2,1)}$ in $\Frob(R_3^{(0,3)};u,v,w)$ must appear in $\Frob(R_3^{(0,2)}\otimes R_3^{(0,1)};u,v,w)$. We compute that
\begin{equation}
\begin{aligned}
    \langle \Frob&(R_3^{(0,2)}\otimes R_3^{(0,1)}; u,v,w), s_{(2,1)}\rangle\\
    &= \sum_{\lambda \vdash 3} \sum_{d=0}^2 \langle \Frob(R_3^{(0,2)};u,v), s_{\lambda} \rangle \langle \Frob(R_3^{(0,1)};w), s_{(3-d,1^d)} \rangle g(\lambda, (3-d,1^d), (2,1)).
\end{aligned}
\end{equation}
Using Rosas' formula for Kronecker coefficients of two hooks, we determine that the only pairs of $\lambda$ and $(3-d,1^d)$ which do not have multiplicity of $0$ are $(3), (2,1)$; $(2,1), (3)$; $(2,1), (2,1)$; $(2,1), (1^3)$; and $(1^3), (2,1)$, which all have multiplicity of $1$. Using Proposition~\ref{prop:KR-sign} and Lemma~\ref{lem:HSlemma}, we obtain
\begin{equation}
    w + ([2]_{u,v} +uv) + ([2]_{u,v} +uv)w + ([2]_{u,v} +uv)w^2 + [3]_{u,v}w.
\end{equation}
By permuting which sets of variables are assigned to $R_n^{(0,2)}$ and $R_n^{(0,1)}$, we can similarly obtain
\begin{equation}
    v + ([2]_{u,w} +uw) + ([2]_{u,w} +uw)v + ([2]_{u,w} +uw)v^2 + [3]_{u,w}v,
\end{equation}
\begin{equation}
    u + ([2]_{w,v} +wv) + ([2]_{w,v} +wv)u + ([2]_{w,v} +wv)u^2 + [3]_{w,v}u.
\end{equation}
Summing the monomials which appear in all three equations, we obtain
\begin{equation}
    u + v + w + uv + uw+ vw + 2uvw.
\end{equation}
Now a computer calculation finds that 
\begin{equation}
    \langle\Frob(R_3^{(0,3)};u,v,w), s_{(2,1)}\rangle = u + v + w + uv + uw+ vw,
\end{equation}
which shows that the bound is not tight in this case due to the monomial $2uvw$. It will require additional work to determine how to cut out the extra monomials in general.
\end{example}

\begin{example}
    Let $n=3$. Suppose we know all of the hook characters in $R_3^{(0,3)}$; in this case $\Frob(R_3^{(0,3)};u,v,w) = (uvw+u^2+uv+v^2+uw+vw+w^2)s_{(1^3)} + (uv+uw+vw+u+v+w)s_{(2, 1)} + s_{(3)}$. Consider that any sign character in $R_3^{(0,4)}$ must appear in $R_3^{(0,3)} \otimes R_3^{(0,1)}$. We compute that
    \begin{equation}
    \begin{aligned}
    \langle \Frob(R_3^{(0,3)}\otimes R_3^{(0,1)};u,v,w,z), s_{(1^3)}\rangle &= \sum_{\lambda \vdash 3} \langle \Frob(R_3^{(0,3)};u,v,w), s_{\lambda'} \rangle \langle \Frob(R_3^{(0,1)};z), s_{\lambda} \rangle\\
    &=\sum_{d=0}^2 \langle \Frob(R_3^{(0,3)};u,v,w), s_{(d+1,1^{2-d})} \rangle z^d,
\end{aligned}
\end{equation}
using Lemma~\ref{lem:HSlemma}. This simplifies to
\begin{equation}\label{eq:(0,4)sign_example}
    u^2 + v^2 + w^2 + z^2 + uv + uw +uz+vw+vz+wz+uvw+uvz+uwz+vwz,
\end{equation}
which is unchanged under changing which sets of variables are assigned to $R_n^{(0,3)}$ and $R_n^{(0,1)}$. In this example, equation~(\ref{eq:(0,4)sign_example}) is equal to
$\langle\Frob(R_3^{(0,4)};u,v,w,z), s_{(1^3)}\rangle$ determined by computer calculation.
\end{example}

\begin{example}
The previous example also gives tight bounds for $n=4$, with graded multiplicity $s_{(3)}(u,v,w,z)+s_{(2,1,1)}(u,v,w,z)+s_{(2,1,1,1)}(u,v,w,z)$, and for $n=5$, with graded multiplicity $s_{(4)}(u,v,w,z) + s_{(3,1,1)}(u,v,w,z)+s_{(3,1,1,1)}(u,v,w,z)+s_{(2,2,1,1)}(u,v,w,z)+s_{(2,2,2,1)}(u,v,w,z)$.
\end{example}

Another approach is to use $\Frob(R_n^{(0,2)} \otimes R_n^{(0,2)};u,v,w,z)$ to bound the sign character in $\Frob(R_n^{(0,4)};u,v,w,z)$ directly. As Kim and Rhoades note, only Schur functions of shapes contained within double hooks occur in $\Frob(R_n^{(0,2)};u,v)$, so the shapes which index sign characters are limited. 

\begin{example}
Let $n=3$. Consider that any sign character in $R_3^{(0,4)}$ must appear in $R_3^{(0,2)}\otimes R_3^{(0,2)}$. We compute that
\begin{align}
    \langle \Frob(R_3^{(0,2)}\otimes R_3^{(0,2)};u,v,w,z), s_{(1^3)}\rangle &= \sum_{\lambda \vdash 3} \langle \Frob(R_3^{(0,2)};u,v), s_{\lambda'} \rangle \langle \Frob(R_3^{(0,2)};w,z), s_{\lambda} \rangle.
\end{align}
In this case, there are only three partitions of $3$: $(1^3), (2,1), (3)$. Using Proposition~\ref{prop:KR-sign}, we get 
\begin{align}
    \langle \Frob(R_3^{(0,2)};u,v), s_{(3)} \rangle \langle \Frob(R_3^{(0,2)};w,z), s_{(1^3)} \rangle = [3]_{w,z},
\end{align}
\begin{align}
    \langle \Frob(R_3^{(0,2)};u,v), s_{(2,1)} \rangle \langle \Frob(R_3^{(0,2)};w,z), s_{(2,1)} \rangle = ([2]_{u,v} + uv)([2]_{w,z} + wz),
\end{align}
\begin{align}
    \langle \Frob(R_3^{(0,2)};u,v), s_{(1^3)} \rangle \langle \Frob(R_3^{(0,2)};w,z), s_{(3)} \rangle = [3]_{u,v}.
\end{align}
Putting these together, we obtain
\begin{equation}
    [3]_{w,z} + ([2]_{u,v} + uv)([2]_{w,z} + wz) + [3]_{u,v}.
\end{equation}
By changing which sets of variables are assigned to each $R_n^{(0,2)}$, we can similarly obtain
\begin{equation}
    [3]_{v,z} + ([2]_{u,w} + uw)([2]_{v,z} + vz) + [3]_{u,w},
\end{equation}
\begin{equation}
    [3]_{u,z} + ([2]_{w,v} + wv)([2]_{u,z} + uz) + [3]_{w,v}.
\end{equation}
In this case, all three are equal to each other, and expand to give
\begin{equation}
    uvwz + uvw + uvz + uwz + vwz + u^2 + uv + v^2 + uw + vw + w^2 + uz + vz + wz + z^2.
\end{equation}
Now a computer calculation finds that 
\begin{equation}
\begin{aligned}
    \langle\Frob&(R_3^{(0,4)};u,v,w,z), s_{(1^3)}\rangle\\ 
    &= uvw + uvz + uwz + vwz + u^2 + uv + v^2 + uw + vw + w^2 + uz + vz + wz + z^2,
\end{aligned}
\end{equation}
which shows that the bound is not tight in this case due to the monomial $uvwz$. It will require additional work to determine how to cut out the extra monomials in general.
\end{example}

\section{One set of bosons and three sets of fermions}\label{sec:one-three}

In this section, we study the sign character in $R_n^{(1,3)}$, the $(1,3)$-bosonic-fermionic coinvariant ring. We first recall the ``Theta conjecture'' of D'Adderio, Iraci, and Vanden Wyngaerd, which expresses the multigraded Frobenius series of $R_n^{(2,2)}$ in terms of certain Theta operators and the nabla operator.

\begin{conjecture}[\!\!{\cite[Conjecture 8.2]{DadderioIraciVandenWyngaerd2021}}]\label{conj:thetaconj}
    For all $n \geq 1$,
\begin{equation} \Frob(R_n^{(2,2)}; q,t;u,v) = \sum_{ k + \ell < n} u^k v^\ell\Theta_{e_k}\Theta_{e_\ell}\nabla e_{n-k-\ell}.\end{equation}
\end{conjecture}

The Theta conjecture specialized at $t=0$ is the following.

\begin{conjecture}[D'Adderio, Iraci, and Vanden Wyngaerd]\label{conj:thetaconj_t=0}
    For all $n \geq 1$,
\begin{equation} \Frob(R_n^{(1,2)}; q;u,v) = \sum_{ k + \ell < n} u^k v^\ell\left(\Theta_{e_k}\Theta_{e_\ell}\nabla e_{n-k-\ell}\right)|_{t=0}.\end{equation}
\end{conjecture}

We recall the following result on the conjectural hook characters in $R_n^{(1,2)}$.

\begin{theorem}[\!\!{\cite[Theorem 8.4]{Lentfer2024}}]\label{thm:lentfer}
    If the Theta conjecture specialized at $t=0$ (Conjecture~\ref{conj:thetaconj_t=0}) is true, then
    \begin{equation} 
    \begin{aligned}
        \langle \Frob&( R_n^{(1,2)};q;u,v), s_{(d+1,1^{n-d-1})}\rangle\\ &= \sum_{k+\ell < n} u^k v^\ell q^{\binom{n-d-k-\ell}{2}}\qbinom{n-1-d}{\ell}_q\qbinom{n-1-k}{d}_q\qbinom{n-1-\ell}{k}_q.
    \end{aligned}
    \end{equation}
\end{theorem}

Now we compute an example.

\begin{example}
    Let $n=3$. Consider that any sign character in $R_3^{(1,3)}$ must appear in $R_3^{(1,2)}\otimes R_3^{(0,1)}$. We compute that
    \begin{equation}
\begin{aligned}
    \langle \Frob&(R_3^{(1,2)}\otimes R_3^{(0,1)};q;u,v,w), s_{(1^3)}\rangle\\ 
    &= \sum_{\lambda \vdash 3} \langle \Frob(R_3^{(1,2)};q;u,v), s_{\lambda'} \rangle \langle \Frob(R_3^{(0,1)};w), s_{\lambda} \rangle\\
    &=\sum_{d=0}^2 \langle \Frob(R_3^{(1,2)};q;u,v), s_{(d+1,1^{2-d})} \rangle w^d\\
    &= (q^3 + vq[2]_q + uq[2]_q + uv[2]_q + v^2 + u^2) + (uv + q[2]_q + v[2]_q + u[2]_q)w + w^2.
\end{aligned}
\end{equation}
In this example, this is equal to
$\langle\Frob(R_3^{(1,3)};q;u,v,w), s_{(1^3)}\rangle$ determined by computer calculation.
\end{example}

\begin{proposition}\label{prop:(1,3)-upper-bound}
If the Theta conjecture specialized at $t=0$ (Conjecture~\ref{conj:thetaconj_t=0}) is true, then we have the following upper-bound:
\begin{equation} 
\begin{aligned}
\langle \Frob&( R_n^{(1,3)};q;u,v,w), s_{(1^n)}\rangle\\ &\leq \sum_{k,\ell,d \geq 0} u^k v^\ell w^d q^{\binom{n-d-k-\ell}{2}}\qbinom{n-1-d}{\ell}_q\qbinom{n-1-k}{d}_q\qbinom{n-1-\ell}{k}_q.
\end{aligned}
\end{equation}
\end{proposition}

\begin{proof}
    Any sign character in $R_n^{(1,3)}$ must appear in $R_n^{(1,2)}\otimes R_n^{(0,1)}$.  Assume the Theta conjecture specialized at $t=0$ (Conjecture~\ref{conj:thetaconj_t=0}) is true. By Theorem~\ref{thm:lentfer}, we compute that
\begin{equation}
\begin{aligned}
    \langle \Frob&(R_n^{(1,2)}\otimes R_n^{(0,1)};q;u,v,w), s_{(1^n)}\rangle\\
    &= \sum_{\lambda \vdash n} \langle \Frob(R_n^{(1,2)};q;u,v), s_{\lambda'} \rangle \langle \Frob(R_n^{(0,1)};w), s_{\lambda} \rangle\\
    &= \sum_{d=0}^{n-1} w^d \langle \Frob(R_n^{(1,2)};q;u,v), s_{(d+1,1^{n-d-1})}\rangle\\
    &= \sum_{d=0}^{n-1} w^d \sum_{k+\ell < n} u^k v^\ell q^{\binom{n-d-k-\ell}{2}}\qbinom{n-1-d}{\ell}_q\qbinom{n-1-k}{d}_q\qbinom{n-1-\ell}{k}_q\\
    &= \sum_{k,\ell,d \geq 0} u^k v^\ell w^d q^{\binom{n-d-k-\ell}{2}}\qbinom{n-1-d}{\ell}_q\qbinom{n-1-k}{d}_q\qbinom{n-1-\ell}{k}_q,
\end{aligned}
\end{equation}
where the last line follows by considering when the $q$-binomial coefficients must be $0$ (specifically, one could take $k+\ell < n$, $\ell+d < n$, and $d +k < n$).
\end{proof}

Based on data for $n \leq 5$, we propose the following conjecture.

\begin{conjecture}\label{conj:one-three}
    \begin{equation} 
    \begin{aligned}
    \langle \Frob&( R^{(1,3)};q;u,v,w), s_{(1^n)}\rangle\\ &= \sum_{k,\ell,d \geq 0} u^k v^\ell w^d q^{\binom{n-d-k-\ell}{2}}\qbinom{n-1-d}{\ell}_q\qbinom{n-1-k}{d}_q\qbinom{n-1-\ell}{k}_q.
    \end{aligned}
    \end{equation}
\end{conjecture}

Notice that upon specializing $q,u,v,w$ all to $1$, the conjecture becomes
\begin{equation}
    \langle \Frob(R^{(1,3)};1;1,1,1), s_{(1^n)}\rangle = \sum_{k,\ell,d \geq 0}\binom{n-1-d}{\ell}\binom{n-1-k}{d}\binom{n-1-\ell}{k}.
\end{equation}

Let $F_n$ be the $n$th Fibonacci number, defined by the initial conditions $F_0=0$, $F_1=1$, and recurrence $F_{n} = F_{n-1} + F_{n-2}$ for all $n \geq 2$. Inspired by a conjecture of Zabrocki on $\Frob(R_n^{(2,1)};q,t;u)$ \cite{Zabrocki2019}, Bergeron made a general conjecture on how the Frobenius series of bosonic-fermionic coinvariant rings could be derived from the Frobenius series of purely bosonic coinvariant rings \cite[Conjecture 1]{Bergeron2020}. While collecting computational evidence, Bergeron conjectured the following.

\begin{conjecture}[\!\!{\cite[Table 3]{Bergeron2020}}]\label{conj:bergeron_fibonacci}
    For all $n \geq 1$,
    \begin{equation} \langle \Frob(R_n^{(1,3)};1;1,1,1), s_{(1^n)}\rangle = \frac{1}{2}F_{3n}.\end{equation}
\end{conjecture}

Beginning at $n=1$, the sequence $\frac{1}{2}F_{3n}$ is $1, 4, 17, 72, 305, 1292, 5473, 23184, \ldots$ (see \cite[Sequence \href{https://oeis.org/A001076}{A001076}]{OEIS}). The two proposed enumerations are connected by the following result.

\begin{proposition}[\!\!\cite{Carlitz}; see \cite{BenjaminRouse} for a bijective proof]\label{prop:fibonacci_identity} For $n \geq 1$,
    \begin{equation}\sum_{k,\ell,d \geq 0}\binom{n-1-d}{\ell}\binom{n-1-k}{d}\binom{n-1-\ell}{k}= \frac{1}{2}F_{3n}.\end{equation}
\end{proposition}

\begin{remark}
    If the Theta conjecture specialized at $t=0$ (Conjecture~\ref{conj:thetaconj_t=0}) is true, and the conjecture of Bergeron (Conjecture~\ref{conj:bergeron_fibonacci}) that $\langle \Frob(R^{(1,3)};1;1,1,1), s_{(1^n)}\rangle = \frac{1}{2}F_{3n}$ is true, then by Proposition~\ref{prop:fibonacci_identity}, the sign character has multiplicity large enough to force equality in the upper bound given in Proposition~\ref{prop:(1,3)-upper-bound}. This would prove Conjecture~\ref{conj:one-three}. 
\end{remark}

\section*{Acknowledgements}

The author would like to thank Fran\c{c}ois Bergeron, Sylvie Corteel, Nicolle Gonz\'alez, and Mark Haiman for helpful conversations. 
Thank you to the referee for comments which helped improve the paper.
The author was partially supported by the National Science Foundation Graduate Research Fellowship DGE-2146752.

\bibliographystyle{amsplain}
\bibliography{biblio}

\end{document}